\newcommand{\TemplateFontSize}{11pt}
\newcommand{\TemplatePaper}{a4paper}
\newcommand{\TemplateLayout}{oneside}
\documentclass[\TemplateFontSize,\TemplatePaper,\TemplateLayout]{amsart}

\newcommand{\DocumentTitle}{Representation Type and Schurian Finiteness of Tits Algebras}

\newcommand{\DocumentAuthor}{Siyuan Chen}

\newcommand{\DocumentInstitution}{Nanyang Technological University}
\newcommand{\DocumentAddress}{School of Physical and Mathematical Sciences, \DocumentInstitution, 21 Nanyang Link, Singapore 637371.}
\newcommand{\DocumentEmail}{siyuan012@e.ntu.edu.sg}
\newcommand{\DocumentSubject}{Representation theory of Tits algebras}
\newcommand{\DocumentKeywords}{Tits algebras, hyperplane arrangements, representation type, Schurian finiteness}

\newif\ifUseBibliography
\UseBibliographytrue

\newif\ifUseIndex
\UseIndexfalse

\usepackage[a4paper,margin=1in]{geometry}
\usepackage{iftex}

\ifPDFTeX
  \usepackage[T1]{fontenc}
  \usepackage[utf8]{inputenc}
  \usepackage{lmodern}
\else
  \usepackage[no-math]{fontspec}
  \defaultfontfeatures{Ligatures=TeX,Scale=MatchLowercase}
  \IfFontExistsTF{Latin Modern Roman}{
    \setmainfont{Latin Modern Roman}
    \setsansfont{Latin Modern Sans}
    \setmonofont{Latin Modern Mono}
  }{
    \setmainfont{lmroman10-regular.otf}[
      SmallCapsFont=lmromancaps10-regular.otf,
      BoldFont=lmroman10-bold.otf,
      ItalicFont=lmroman10-italic.otf,
      BoldItalicFont=lmroman10-bolditalic.otf
    ]
    \setsansfont{lmsans10-regular.otf}[
      BoldFont=lmsans10-bold.otf,
      ItalicFont=lmsans10-oblique.otf,
      BoldItalicFont=lmsans10-boldoblique.otf
    ]
    \setmonofont{lmmono10-regular.otf}[
      BoldFont=lmmonolt10-bold.otf,
      ItalicFont=lmmono10-italic.otf,
      BoldItalicFont=lmmonolt10-boldoblique.otf
    ]
  }

  \ifXeTeX
    \IfFileExists{xeCJK.sty}{
      \usepackage[LoadFandol=false]{xeCJK}
      \setCJKmainfont{FandolSong-Regular.otf}[
        BoldFont=FandolSong-Bold.otf,
        ItalicFont=FandolKai-Regular.otf,
        BoldItalicFont=FandolKai-Regular.otf
      ]
      \setCJKmonofont{FandolFang-Regular.otf}
    }{}
  \fi
\fi

\usepackage{microtype}
\usepackage{setspace}
\usepackage{enumitem}
\usepackage{xcolor}
\usepackage{graphicx}
\usepackage{float}
\usepackage{caption}
\usepackage{subcaption}
\usepackage{booktabs}
\usepackage{array}
\usepackage{mathtools}
\usepackage{amssymb,amsmath,amsthm,amscd}
\usepackage{mathrsfs}
\usepackage{hyperref}
\usepackage{aliascnt}
\usepackage[nameinlink,capitalize,noabbrev]{cleveref}
\usepackage{bookmark}
\usepackage[numbers,sort&compress]{natbib}
\usepackage[all,cmtip]{xy}
\ifUseIndex
\usepackage[noautomatic]{imakeidx}
\fi

\graphicspath{{figures/}{figures/tikz/}{figures/diagrams/}}

\ifUseIndex
  \makeindex
\fi

\hypersetup{
    colorlinks=true,
    linkcolor=blue!60!black,
    citecolor=green!40!black,
    urlcolor=magenta!60!black,
    pdftitle={\DocumentTitle},
    pdfauthor={\DocumentAuthor},
    pdfsubject={\DocumentSubject},
    pdfkeywords={\DocumentKeywords}
}

\setlist[enumerate,1]{label=\textup{(\arabic*)},ref=\arabic*}
\setlist[enumerate,2]{label=\textup{(\alph*)},ref=\alph*}

\providecommand{\doi}[1]{\url{https://doi.org/#1}}

\theoremstyle{plain}
\newtheorem{theorem}{Theorem}[section]

\newaliascnt{lemma}{theorem}
\newtheorem{lemma}[lemma]{Lemma}
\aliascntresetthe{lemma}

\theoremstyle{definition}
\newaliascnt{definition}{theorem}
\newtheorem{definition}[definition]{Definition}
\aliascntresetthe{definition}

\crefname{subsection}{Subsection}{Subsections}
\Crefname{subsection}{Subsection}{Subsections}

\newcommand{\K}{\mathbb{K}}

\newcommand{\R}{\mathbb{R}}

\newcommand{\cA}{\mathcal{A}}
\newcommand{\cB}{\mathcal{B}}
\newcommand{\cC}{\mathcal{C}}
\newcommand{\cF}{\mathcal{F}}

\renewcommand{\a}{\mathfrak{a}}

\DeclareMathOperator{\Hom}{Hom}
\DeclareMathOperator{\End}{End}

\DeclareMathOperator{\id}{id}

\definecolor{myblue}{RGB}{25,85,160}
\definecolor{mygreen}{RGB}{30,120,80}
\definecolor{myred}{RGB}{160,40,40}
\definecolor{mygray}{RGB}{90,90,90}

\title[Representation Type and Schurian Finiteness]{\DocumentTitle}
\author{\DocumentAuthor}
\address[\DocumentAuthor]{\DocumentAddress}
\email{\href{mailto:\DocumentEmail}{\DocumentEmail}}
\date{}

\begin{document}

\begin{abstract}
We determine the representation type and Schurian finiteness of the Tits
algebra of a finite central real hyperplane arrangement over an
algebraically closed field. The algebra is representation-finite when the
arrangement has at most two hyperplanes, tame when it has three, and wild
when it has at least four. Moreover, it is Schurian-finite if and only if
the arrangement has at most two hyperplanes. Both classifications hold over
algebraically closed fields of arbitrary characteristic. The proof combines Saliola's
bound-quiver presentation with explicit module constructions, convex
subquivers, essentialization, and deletion. In the $\R^2$ case, we describe all
eleven indecomposable modules for two lines and construct a full exact
embedding of modules over a free algebra for at least four lines.
\end{abstract}
\maketitle

\section{Introduction}\label{ch:intro}

The face semigroup of a real hyperplane arrangement encodes the geometry of
its faces by a noncommutative multiplication, and its semigroup algebra is
called the Tits algebra. One motivation for studying this algebra comes from
random walks on chambers: Brown and Diaconis~\citep{brownDiaconis1998}
used the face action to determine stationary distributions and study the
spectra and convergence of these walks. From a representation-theoretic
perspective, Saliola~\citep{saliola2009face} determined the quiver and
relations of the face semigroup algebra, constructed its primitive
idempotents, and proved that it is Koszul and depends only on the
intersection lattice. These results turn questions about an algebra defined
geometrically into explicit problems about a finite bound quiver.
They also make it natural to ask how the combinatorics of an arrangement
controls the complexity of its module category.

Further motivation comes from the representation theory of Solomon
descent algebras. For the braid arrangement, the symmetric-group invariant
subalgebra of the Tits algebra is anti-isomorphic to the Solomon descent
algebra~\citep{saliola2009face}. Schocker~\citep{schocker2006SolomonTits}
studied the full Solomon--Tits algebra of the symmetric group, describing
its simple and projective indecomposable modules, Cartan invariants, and
Ext-quiver. More recently, Erdmann and
Lim~\citep{erdmannLim2025descent} classified the representation type of
the type-$\mathbb A$ descent algebras in positive characteristic.
In a subsequent preprint~\citep{erdmannLim2026descent}, they determine
the representation type of the Solomon descent algebras for the other
irreducible finite Coxeter types, with type $\mathbb E_8$ remaining open.
These developments motivate the corresponding classification problem for
the full Tits algebra of an arbitrary central arrangement and the
comparison with Schurian finiteness. In particular, one can ask how both
properties depend on the number and relative positions of the hyperplanes.

In this paper, we study two such measures of complexity: representation
type and Schurian finiteness. The finite--tame--wild classification concerns
all finite-dimensional indecomposable modules, whereas Schurian finiteness
concerns those nonzero modules whose endomorphism algebra consists only of
scalars. We use Saliola's presentation to classify the rank-two algebras,
including an explicit list of eleven indecomposables for two lines and a
full exact embedding of free-algebra modules for at least four lines.
For three independent hyperplanes, the algebra is the incidence algebra of
the Boolean lattice $B_3$, whose tameness is recalled by
Chapoton and Rognerud~\citep[Lemma~3.5]{chapotonRognerud2019}. Together with the four-hyperplane
cases, essentialization and deletion yield the general classification.
In particular, every central arrangement with at least four distinct
hyperplanes has a wild Tits algebra, while Schurian finiteness holds
precisely when there are at most two hyperplanes. The tame cases already
have infinitely many Schurian modules, so the two classifications capture
different features of the module category.

The proof proceeds from small configurations to general arrangements.
\Cref{ch:basic_results} develops the bound-quiver and quotient tools,
and \cref{ch:rank_two} treats the planar algebras by explicit module
constructions. \cref{ch:dimension_three,ch:dimension_four} then use the geometry of three and four
hyperplanes to identify suitable convex subquivers. Finally,
\cref{ch:general} separates two complementary reductions:
essentialization removes redundant ambient directions without changing the
Tits algebra, whereas deletion produces quotient algebras attached to
smaller arrangements. These reductions allow the wildness and Schurian
infiniteness detected in the small configurations to pass to arbitrary
central arrangements, explaining why a finite list of geometric cases
suffices for the classification.

We can now state the main results for essential arrangements of dimension
at least two; the low-rank and nonessential cases are included in
\cref{ch:general}.

\begin{samepage}
\begin{theorem}\label{thm:main}
Let $\K$ be an algebraically closed field, and let
$\cA=\{H_1,\ldots,H_m\}$ be a central essential arrangement of $m$
distinct hyperplanes in $\R^n$, where $n\geq2$. Write
$T_{\cA}:=\K[\cF(\cA)]$ for its Tits algebra.
\begin{enumerate}
    \item Suppose that $n=2$. Then $m\geq2$, and:
    \begin{enumerate}
        \item If $m=2$, then $T_{\cA}$ is representation-finite.
        \item If $m=3$, then $T_{\cA}$ is of tame representation type.
        \item If $m\geq4$, then $T_{\cA}$ is of wild representation type.
    \end{enumerate}
    \item Suppose that $n=3$.
    \begin{enumerate}
        \item If $m=3$, then $T_{\cA}$ is of tame representation type.
        \item If $m\geq4$, then $T_{\cA}$ is of wild representation type.
    \end{enumerate}
    \item Suppose that $n\geq4$. Then $T_{\cA}$ is of wild
    representation type.
\end{enumerate}
Furthermore, $T_{\cA}$ is Schurian-finite if and only if $n=2$ and $m=2$.
\end{theorem}
\end{samepage}
\section{Basic Results}\label{ch:basic_results}

Throughout, the base field $\K$ is algebraically closed. For a
$\K$-algebra $A$, we write $A\text{-mod}$ for the category of
finite-dimensional left $A$-modules. We first recall bound quiver
representations and Schurian finiteness, then
introduce face semigroups and the bound-quiver presentation of Tits algebras.

\subsection{Bound quiver algebras, representations, and Schurian finiteness}
For background on finite-dimensional algebras, bound quivers and their
representations, see Assem, Simson and Skowro\'nski~\citep[Chapters~II--III]{AssemSimsonSkowronski2006},
Benson~\citep[Chapter~4]{Benson1991}, and
Schiffler~\citep[Chapters~1 and~5]{Schiffler2014}.

Let $Q=(Q_0,Q_1)$ be a finite quiver. Its path algebra $\K Q$ has as a
$\K$-basis all oriented paths in $Q$, including the stationary path $e_i$ at
each vertex $i$, with multiplication given by concatenation of composable
paths and zero otherwise. Let $J$ be the ideal of $\K Q$ generated by the arrows. An
ideal $I\subseteq \K Q$ is called \emph{admissible} if
$J^\ell\subseteq I\subseteq J^2$ for some integer $\ell\geq2$. The quotient $\K Q/I$ is then called a \emph{bound quiver
algebra}, and the pair $(Q,I)$ is called a \emph{bound quiver}.

A finite-dimensional representation $M$ of $(Q,I)$ assigns a vector space
$M_i$ to each vertex $i\in Q_0$ and a linear map
$M_a:M_i\to M_j$ to each arrow $a:i\to j$. These maps must satisfy the
relations in $I$: if $\sum_p c_p p\in I$ is a relation whose paths have a
common source and target, then
$\sum_p c_p M_p=0$, where $M_p$ denotes the composition of the maps along
the path $p$. Morphisms are collections of vertexwise linear maps commuting
with all arrow maps. In this way, one obtains the standard equivalence
\[
    \operatorname{rep}_{\K}(Q,I)
    \simeq
    (\K Q/I)\text{-mod}.
\]

Let $A$ be a finite-dimensional $\K$-algebra.  A nonzero finite-dimensional
$A$-module $M$ is called a \emph{brick} if its endomorphism algebra
$\End_A(M)$ is a division algebra.  Since $\K$ is algebraically closed, this
is equivalent in the present setting to $\End_A(M)\cong\K$.
A module satisfying this condition is also called \emph{Schurian}.
The algebra $A$ is called \emph{Schurian-finite} if there are only finitely
many isomorphism classes of finite-dimensional Schurian $A$-modules; it is
called \emph{Schurian-infinite} otherwise.  Thus, over $\K$,
Schurian-finiteness is the same as brick-finiteness.

Schurian finiteness is closely related to $\tau$-tilting
theory~\citep{AdachiIyamaReiten2014}.
Demonet, Iyama and Jasso~\citep[Theorem~4.2]{DemonetIyamaJasso2019}
prove that a finite-dimensional algebra is $\tau$-tilting finite if and
only if it has finitely many isomorphism classes of bricks.
Consequently, over the algebraically closed field $\K$,
Schurian finiteness and $\tau$-tilting finiteness are equivalent.

We will use the following quotient-inheritance result.
\begin{lemma}\label{lem:quotient_inheritance}
Let $A$ be a finite-dimensional $\K$-algebra, let $\a$ be a two-sided
ideal of $A$, and set $B:=A/\a$. Then:
\begin{enumerate}
    \item If $B$ is of wild representation type, then so is $A$.
    \item If $B$ is Schurian-infinite, then so is $A$.
\end{enumerate}
\end{lemma}

\begin{proof}
The inflation functor
$\operatorname{Inf}:B\text{-mod}\to A\text{-mod}$ is exact and fully
faithful: the $A$-action on a $B$-module is obtained by composing its action
with the quotient map $A\to B$.
If $B$ is wild, its defining representation embedding has source the
category of finite-dimensional modules over the free associative algebra
$\K\langle x,y\rangle$. Composing it with inflation gives
\[
    \K\langle x,y\rangle\text{-mod}
    \longrightarrow B\text{-mod}
    \xrightarrow{\operatorname{Inf}} A\text{-mod}.
\]
The composite is again a representation embedding,
since it is exact, preserves indecomposability, and reflects isomorphism
classes. Thus $A$ is wild.

For Schurian infiniteness, inflation preserves endomorphism algebras and
reflects isomorphism classes. Consequently, an infinite family of pairwise
non-isomorphic Schurian $B$-modules gives such a family of $A$-modules.
\end{proof}

\subsection{Hyperplane arrangements and Tits algebras}
For standard terminology and background on hyperplane arrangements and
their intersection lattices, we refer to
Stanley~\citep{Stanley2007Arrangements}.

A \emph{central real hyperplane arrangement} in $\R^n$
is a finite set $\cA=\{H_1,\ldots,H_m\}$ of distinct linear
hyperplanes. It is \emph{essential} if
$\bigcap_{i=1}^m H_i=\{0\}$. Its \emph{intersection lattice}
$L(\cA)$ consists of all intersections of subfamilies, including $\R^n$
as the intersection of the empty subfamily, ordered by reverse inclusion.
We write $m=|\cA|$ for the number of hyperplanes and
\[
    r=\operatorname{rk}(\cA)
    :=\operatorname{codim}_{\R^n}\left(\bigcap_{H\in\cA}H\right)
\]
for its rank. Thus $n$ is the ambient dimension, and $r=n$ precisely
when the arrangement is essential. Essentialization preserves $m$ and
replaces the ambient space by one of dimension $r$; see
\cref{lem:essentialization}.

Choose a defining linear form for each $H_i$, and let $H_i^+$ and $H_i^-$
be its open positive and negative half-spaces, with $H_i^0=H_i$.
A \emph{face} is a nonempty set of the form
\[
    F=\bigcap_{i=1}^m H_i^{\epsilon_i},
    \qquad \epsilon_i\in\{+,-,0\}.
\]
In particular, the full-dimensional faces are the \emph{chambers}, namely
the connected components of $\R^n\setminus\bigcup_{i=1}^m H_i$. The faces form a
poset under the closure order $F\leq G$ if $F\subseteq\overline G$.

For faces $F$ and $G$ with sign vectors $(\epsilon_i)$ and $(\delta_i)$,
their \emph{Tits product} is the face with sign vector $(\gamma_i)$, where
\[
    \gamma_i=
    \begin{cases}
        \epsilon_i,&\epsilon_i\neq0,\\
        \delta_i,&\epsilon_i=0.
    \end{cases}
\]
Geometrically, it is the face containing $(1-t)f+tg$ for $f\in F$,
$g\in G$, and sufficiently small $t>0$.
This multiplication makes the set $\cF(\cA)$ of faces a semigroup.
The \emph{Tits algebra} is the semigroup algebra
$T_{\cA}:=\K[\cF(\cA)]$; see Saliola~\citep{saliola2009face}.

\subsection{Saliola's bound-quiver presentation and convex subquivers}
\label{subsec:Saliola}

The bound-quiver presentation of Saliola~\citep{saliola2009face} is the main link between
the geometry of an arrangement and the representation theory of its Tits
algebra.

Let $(P,\leq)$ be a finite poset.  An element $x\in P$ is said to
\emph{cover} $y\in P$, written $x\gtrdot y$,
if $y<x$ and there is no $z\in P$ such that $y<z<x$.  The Hasse diagram
of $P$ has one vertex for each element of $P$ and one edge between $x$ and
$y$ for each cover relation $x\gtrdot y$; when the diagram is oriented
downward, this edge is written $x\to y$.

In this paper $L(\cA)$ is ordered by reverse inclusion, whereas Saliola's
bound-quiver presentation uses the ordinary inclusion order, namely
$L(\cA)^{\mathrm{op}}$. Thus, in the bound-quiver discussion below, the notation
$X\gtrdot Y$ means concretely that $Y\subsetneq X$ and that there is no
$Z\in L(\cA)$ with $Y\subsetneq Z\subsetneq X$.  For an intersection
lattice this is equivalent to $\dim X-\dim Y=1$, and the corresponding arrow
in the quiver is $X\to Y$.

The \emph{Boolean lattice} $B_d$ is the set of all subsets of
$\{1,\ldots,d\}$, ordered by inclusion. A cover relation adds one
element, so its Hasse diagram is the $d$-dimensional cube; in particular,
$B_1$ is a two-element chain and $B_3$ is the ordinary cube. For an
arrangement of $d$ hyperplanes with linearly independent defining forms,
the map $S\mapsto\bigcap_{i\in S}H_i$ identifies $B_d$ with
$L(\cA)$ in the reverse-inclusion order.
For a finite poset $P$, its \emph{incidence algebra} $I(P)$ has basis
$e_{xy}$ for $x\leq y$ in $P$, with multiplication
$e_{xy}e_{yz}=e_{xz}$ and zero when the middle indices differ.
We will use the incidence algebras
$I(B_d)$ in the low-dimensional cases.

\begin{theorem}[Saliola's bound-quiver presentation]
\label{lem:saliola_bound_quiver}
Let $\cA$ be a finite central real hyperplane arrangement, let $\cF(\cA)$ be
its face semigroup, and let $T_{\cA}:=\K[\cF(\cA)]$.
Define a quiver $Q_{\cA}$ whose vertex set is the intersection lattice
$L(\cA)$.  For $X,Y\in L(\cA)$, there is one arrow
$X\longrightarrow Y$
precisely when $Y\subsetneq X$ and $\dim X-\dim Y=1$.  For every pair
$Z\subsetneq X$ with $\dim X-\dim Z=2$, set
\[
    \rho_{X,Z}
    :=\sum_{\substack{Y\in L(\cA)\\ Z\subsetneq Y\subsetneq X}}
      (X\longrightarrow Y\longrightarrow Z).
\]
Then
\[
    T_{\cA}
    \cong
    \K Q_{\cA}\Big/
    \left\langle \rho_{X,Z}:Z\subsetneq X,\ \dim X-\dim Z=2\right\rangle.
\]
The displayed relations form a minimal generating set for the defining ideal.
In particular, the isomorphism type of $T_{\cA}$ depends only on the
intersection lattice $L(\cA)$.
\end{theorem}

\begin{proof}
Saliola orders the intersection lattice by inclusion and orients each cover as
$X\to Y$ when $X\gtrdot Y$. By
\citep[Corollary~8.4]{saliola2009face}, this oriented Hasse diagram is
the Gabriel quiver of the face semigroup algebra.
The relations $\rho_{X,Z}$ and their minimality are given in
\citep[Proposition~8.5]{saliola2009face}, and the final invariance
statement is \citep[Corollary~8.6]{saliola2009face}.
Writing the arrows directly in terms of
subspace containment gives the stated convention, independently of whether
$L(\cA)$ is displayed using inclusion or reverse inclusion.
\end{proof}

The following consequence will be used to detect wildness and Schurian
infiniteness on smaller quivers. A set of vertices of a poset is called
\emph{convex} if it contains every element lying between any two of its
elements.
\begin{lemma}\label{cor:convex-subquiver}
Let $S\subseteq L(\mathcal A)$ be a convex set of vertices, and let
$Q_S$ be the full subquiver of $Q_{\mathcal A}$ induced by $S$.
Suppose that $S$ contains no interval of length two. Then the full
subcategory of $T_{\mathcal A}\text{-mod}$ consisting of modules
supported on $S$ is equivalent to $\K Q_S\text{-mod}$, and the inclusion
is exact and fully faithful.

In particular, if $\K Q_S$ is wild, then $T_{\mathcal A}$ is wild.
Likewise, if $\K Q_S$ is Schurian-infinite, then $T_{\mathcal A}$ is
Schurian-infinite.
\end{lemma}
\begin{proof}
Let $e_{S^c}:=\sum_{X\notin S}e_X$ and consider the quotient
$T_{\mathcal A}/T_{\mathcal A}e_{S^c}T_{\mathcal A}$.
Its module category identifies with the full subcategory of
$T_{\mathcal A}\text{-mod}$ consisting of modules supported on $S$.

Since $S$ is convex, every path in $Q_{\mathcal A}$ whose endpoints
belong to $S$ is entirely contained in $S$. Hence the above quotient
is the bound quiver algebra obtained by restricting
$(Q_{\mathcal A},I_{\mathcal A})$ to $S$.
The defining relations of $T_{\mathcal A}$ are supported on intervals
of length two. Since $S$ contains no such interval, no defining
relation remains after restriction. Therefore
\[
    T_{\mathcal A}/T_{\mathcal A}e_{S^c}T_{\mathcal A}
    \cong \K Q_S.
\]
The assertion now follows from \cref{lem:quotient_inheritance}.
\end{proof}
\section{\texorpdfstring{The \(\R^2\) Case}{The R2 Case}}
\label{ch:rank_two}

In this section the ambient dimension is fixed at $n=2$.
Let $\K$ be an algebraically closed field, and for $m\geq 2$ let
$\cA_m$ be a central essential arrangement of $m$ distinct lines in $\R^2$.
Its intersection lattice has a unique maximal interval of length two.
By \cref{lem:saliola_bound_quiver}, its Tits algebra is
\[
    A_m:=T_{\cA_m}
    \cong
    \K Q_m\Big/\left\langle\sum_{i=1}^m\beta_i\alpha_i\right\rangle,
\]
where $Q_m$ is the $m$-arm quiver
\[
\xymatrix@R=3.0em@C=3.2em{
    && s\ar[dll]_{\alpha_1}\ar[dl]^{\alpha_2}\ar[dr]_{\alpha_{m-1}}\ar[drr]^{\alpha_m} && \\
    x_1\ar[drr]_(.32){\beta_1}
        & x_2\ar[dr]_(.32){\beta_2}
        & {\cdots}
        & x_{m-1}\ar[dl]_(.32){\beta_{m-1}} & x_m\ar[dll]^(.32){\beta_m} \\
    && t &&
}
\]
There is one arm $s\xrightarrow{\alpha_i}x_i\xrightarrow{\beta_i}t$
for each $1\leq i\leq m$. The diagram is schematic; the cases $m=2$
and $m=3$ are displayed separately below.
For representation type, we classify these algebras in the three cases $m=2$, $m=3$, and $m\geq4$,
and then determine their Schurian finiteness.

\subsection{\texorpdfstring{The case \(m=2\)}{The case m=2}}

For $m=2$, the presentation above specializes to the quiver $Q_2$:
\[
\xymatrix@R=1.5em@C=1.7em{
    & s\ar[dl]_{\alpha_1}\ar[dr]^{\alpha_2} & \\
    x_1\ar[dr]_{\beta_1} && x_2\ar[dl]^{\beta_2} \\
    & t &
}
\]
Thus $A_2=\K Q_2/\langle\beta_1\alpha_1+\beta_2\alpha_2\rangle$. We now prove that
this algebra is representation-finite.

\begin{theorem}\label{thm:rank_two_finite}
The algebra $A_2$ has exactly eleven isomorphism classes of indecomposable
finite-dimensional modules. To describe them, write
$U=M_s$, $V=M_{x_1}$, $W=M_{x_2}$, and $Z=M_t$, and let
$a,b,c,d$ be the maps along $\alpha_1,\alpha_2,\beta_1,\beta_2$, respectively.
The diagram and table below use the original arrow maps, which satisfy
$ca+db=0$, to describe the eleven representations.
\begin{center}
\upshape
\begin{minipage}[c]{0.27\linewidth}
\centering
\(\xymatrix@R=1.4em@C=1.4em{
    & U\ar[dl]_a\ar[dr]^b & \\
    V\ar[dr]_c && W\ar[dl]^d \\
    & Z &
}\)
\end{minipage}\hfill
\begin{minipage}[c]{0.70\linewidth}
\centering
\setlength{\tabcolsep}{0.7em}
\begin{tabular}{@{}l*{8}{>{$}c<{$}}@{}}
\toprule
& \multicolumn{4}{c}{Vertex spaces}
& \multicolumn{4}{c}{Arrow maps} \\
\cmidrule(lr){2-5}\cmidrule(l){6-9}
Module & U & V & W & Z & a & b & c & d \\
\midrule
$S_s$     & \K & 0   & 0   & 0   & 0 & 0  & 0 & 0 \\
$S_{x_1}$ & 0   & \K & 0   & 0   & 0 & 0  & 0 & 0 \\
$S_{x_2}$ & 0   & 0   & \K & 0   & 0 & 0  & 0 & 0 \\
$S_t$     & 0   & 0   & 0   & \K & 0 & 0  & 0 & 0 \\
$M_{s1}$  & \K & \K & 0   & 0   & 1 & 0  & 0 & 0 \\
$M_{s2}$  & \K & 0   & \K & 0   & 0 & 1  & 0 & 0 \\
$M_{1t}$  & 0   & \K & 0   & \K & 0 & 0  & 1 & 0 \\
$M_{2t}$  & 0   & 0   & \K & \K & 0 & 0  & 0 & 1 \\
$M_{s12}$ & \K & \K & \K & 0   & 1 & 1  & 0 & 0 \\
$M_{12t}$ & 0   & \K & \K & \K & 0 & 0  & 1 & 1 \\
$F$       & \K & \K & \K & \K & 1 & -1 & 1 & 1 \\
\bottomrule
\end{tabular}
\end{minipage}
\end{center}
Here $1$ denotes $\id_{\K}$, and $0$ denotes a zero space or map.
Thus $A_2$ is representation-finite.
\end{theorem}
\begin{proof}
In this proof, we decompose the module step by step to identify all indecomposable summands.

Write an $A_2$-module as $M=(U,V,W,Z;a,b,c,d)$ with $ca+db=0$, as above.
Replacing $b$ by $-b$ gives an equivalent category of commutative-square
representations. In the rest of the proof, we use $a,b,c,d$ for the resulting
maps, so that $ca=db$.
Let $h:=ca=db:U\to Z$ and $r=\operatorname{rank}(h)$. We may decompose
$U=U_1\oplus U_0$ and $Z=Z_1\oplus Z_0$
such that $h|_{U_1}:U_1\xrightarrow{\sim} Z_1$ is an isomorphism and $h(U_0)=0$.

Moreover, set $V_1:=a(U_1)$ and $W_1:=b(U_1)$.
Since $h|_{U_1}$ is injective, both $a|_{U_1}$ and $b|_{U_1}$ are injective. Since
$c\,a|_{U_1}=h|_{U_1}$ and $d\,b|_{U_1}=h|_{U_1}$
are isomorphisms onto $Z_1$, both $c|_{V_1}:V_1\to Z_1$ and
$d|_{W_1}:W_1\to Z_1$ are also isomorphisms.

Let $\pi_1:Z\to Z_1$ be the projection, and define $V_0:=\ker(\pi_1 c)$
and $W_0:=\ker(\pi_1 d)$. Then $V=V_1\oplus V_0$ and $W=W_1\oplus W_0$
because $\pi_1c|_{V_1}$ and $\pi_1d|_{W_1}$ are isomorphisms onto $Z_1$. Moreover,
\[
a(U_0)\subseteq V_0,
\qquad
b(U_0)\subseteq W_0,
\qquad
c(V_0)\subseteq Z_0,
\qquad
d(W_0)\subseteq Z_0,
\]
since $h(U_0)=0$.
Thus $M=M_1\oplus M_0$, where $M_i=(U_i,V_i,W_i,Z_i)$ for $i=0,1$.
Now $U_1\cong V_1\cong W_1\cong Z_1\cong \K^r$, and after choosing bases, all four maps in $M_1$ become identities. Hence
$M_1\cong F^{\oplus r}$,
where $F$ denotes the full commutative square
\[
F:\qquad
\xymatrix@R=2.6em@C=3.4em{
    {\K}\ar[r]^1\ar[d]_1 & {\K}\ar[d]^1 \\
    {\K}\ar[r]_1 & {\K}
}
\]
Thus every module decomposes as
\[
M\cong F^{\oplus r}\oplus M_0,
\]
where the residual summand $M_0$ satisfies $ca=db=0$.

We now write the residual module as $M=(U,V,W,Z;a,b,c,d)$,
so $ca=db=0$, and decompose it further. Set
$V^{\uparrow}:=a(U)$ and $W^{\uparrow}:=b(U)$.
Since $ca=0$ and $db=0$, we have $c(V^{\uparrow})=0$ and $d(W^{\uparrow})=0$.
Choose a linear subspace $V^{\downarrow}\subseteq V$ complementary to
$V^{\uparrow}=a(U)$, and similarly choose $W^{\downarrow}\subseteq W$
complementary to $W^{\uparrow}=b(U)$. In other words,
$V=V^{\uparrow}\oplus V^{\downarrow}$ and
$W=W^{\uparrow}\oplus W^{\downarrow}$; these are choices of complements,
not canonical subspaces.

Then $M=M^{\uparrow}\oplus M^{\downarrow}$, where
\[
\begin{aligned}
M^{\uparrow}&:=(U,V^{\uparrow},W^{\uparrow},0;a,b,0,0),\\
M^{\downarrow}&:=(0,V^{\downarrow},W^{\downarrow},Z;
0,0,c|_{V^{\downarrow}},d|_{W^{\downarrow}}).
\end{aligned}
\]
Thus $M^{\downarrow}$ has vertex spaces $0$, $V^{\downarrow}$,
$W^{\downarrow}$, and $Z$ at $s$, $x_1$, $x_2$, and $t$, respectively.
Its maps out of $s$ are zero, and its maps into $t$ are the restrictions of
$c$ and $d$ to the chosen complements. In $M^{\uparrow}$, the maps $a$ and
$b$ are viewed as maps into their images $V^{\uparrow}$ and $W^{\uparrow}$.
So it remains to classify the upper fork $M^{\uparrow}$ and lower fork $M^{\downarrow}$ separately.

We consider the upper fork first. Let $M^{\uparrow}=(U,V,W,0;a,b,0,0)$ be an upper fork.

Consider the linear map $f=(a,b):U\to V\oplus W$.
Let $K=\ker f$. Choose $U=K\oplus U'$. Then the representation splits off a summand $S_s^{\oplus \dim K}$, so it suffices to consider the case where $f$ is injective. Identify $U'$ with its image
$L:=f(U')\subseteq V\oplus W$.
Let $L_1:=L\cap V$ and $L_2:=L\cap W$, and choose a complement $L_3$ such that
$L=L_1\oplus L_2\oplus L_3$.
Since $L_3\cap V=0$ and $L_3\cap W=0$, the projections of $L_3$ to $V$ and $W$ are injective. Denote their images by $V_3\subseteq V$ and $W_3\subseteq W$; then $L_3\cong V_3\cong W_3$.
Choose complements $V=L_1\oplus V_3\oplus V_0$ and $W=L_2\oplus W_3\oplus W_0$.
With respect to these decompositions, the representation splits as
\[
S_s^{\oplus \dim K}
\oplus M_{s1}^{\oplus \dim L_1}
\oplus M_{s2}^{\oplus \dim L_2}
\oplus M_{s12}^{\oplus \dim L_3}
\oplus S_{x_1}^{\oplus \dim V_0}
\oplus S_{x_2}^{\oplus \dim W_0}.
\]
Therefore every representation of the upper fork is a direct sum of copies of
$S_s$, $S_{x_1}$, $S_{x_2}$, $M_{s1}$, $M_{s2}$, and $M_{s12}$.

For the lower fork, let $N=(0,V,W,Z;0,0,c,d)$, with diagram
\[
\xymatrix@R=1.3em@C=1.8em{
    V\ar[dr]^c && W\ar[dl]_d \\
    & Z &
}
\]
Applying the vector-space duality $D=\Hom_{\K}(-,\K)$
gives the representation
\[
\xymatrix@R=1.3em@C=1.8em{
    & Z^*\ar[dl]_{c^*}\ar[dr]^{d^*} & \\
    V^* && W^*
}
\]
of the opposite quiver. Explicitly, $c^*:Z^*\to V^*$ sends
$\lambda$ to $\lambda\circ c$, and similarly $d^*(\lambda)=\lambda\circ d$.
If $f=(f_V,f_W,f_Z):N\to N'$ is a morphism, then
$f_Zc=c'f_V$ and $f_Zd=d'f_W$. Taking duals gives
$c^*f_Z^*=f_V^*(c')^*$ and $d^*f_Z^*=f_W^*(d')^*$,
so the vertexwise dual maps define a morphism $D(f):D(N')\to D(N)$.
Thus $D$ reverses both arrows and morphisms.

Identifying this opposite lower fork with the upper
fork via $Z^*\longleftrightarrow U$, $V^*\longleftrightarrow V$, and
$W^*\longleftrightarrow W$,
the preceding upper-fork classification applies. To justify transferring
the decomposition back, finite-dimensionality ensures that the evaluation
maps $X\to X^{**}$ are isomorphisms. They commute with the arrow maps and
therefore give a natural isomorphism $N\cong D^2(N)$. Hence $D$ is a
contravariant equivalence.
Moreover, $D(P\oplus Q)\cong D(P)\oplus D(Q)$, and the dual of a nonzero
representation is nonzero. Thus a nontrivial decomposition of $N$ gives one
of $D(N)$, and conversely a decomposition $D(N)\cong P\oplus Q$ gives
$N\cong D(P)\oplus D(Q)$. This proves that $D$ preserves and reflects
indecomposability. Applying $D$ to the upper-fork decomposition of $D(N)$
now shows that $N$ is a direct sum of
copies of $S_{x_1}$, $S_{x_2}$, $S_t$, $M_{1t}$, $M_{2t}$, and $M_{12t}$.
More explicitly, under the above identification the upper- and lower-fork
indecomposables correspond as
\[
D(S_s)\longleftrightarrow S_t,
\qquad
D(S_{x_i})\longleftrightarrow S_{x_i},
\qquad
D(M_{si})\longleftrightarrow M_{it}\quad(i=1,2),
\qquad
D(M_{s12})\longleftrightarrow M_{12t}.
\]
Consequently every $A_2$-module is a direct sum of copies of
\[
S_s,
\ S_{x_1},
\ S_{x_2},
\ S_t,
\ M_{s1},
\ M_{s2},
\ M_{1t},
\ M_{2t},
\ M_{s12},
\ M_{12t},
\ F.
\]
Each displayed representation has one-dimensional nonzero vertex spaces
and connected support joined by nonzero arrow maps. Hence its
endomorphisms are precisely scalar multiplications, so it is indecomposable.
Their supports are distinct, so they are pairwise non-isomorphic. Therefore
these are exactly the eleven indecomposables, and $A_2$ is
representation-finite.
\end{proof}

\subsection{\texorpdfstring{The case \(m=3\)}{The case m=3}}
\label{sec:m3-tits-canonical-222}

Let $\cA=\{H_1,H_2,H_3\}$
be a central essential arrangement of three distinct lines in $\R^2$, and let
$T_{\cA}$ denote its Tits algebra. The point of this case is that the
bound-quiver presentation in \cref{lem:saliola_bound_quiver} is exactly the classical
three-arm canonical algebra of type $(2,2,2)$.

Historically, Ringel introduced these three-arm algebras
in his study of tame algebras and integral quadratic forms
\citep{Ringel1984}; later, Geigle--Lenzing placed these algebras in the general
framework of weighted projective lines and canonical algebras
\citep{GeigleLenzing1987}.

\begin{definition}[Ringel's three-arm canonical algebra]
\label{def:ringel_three_arm}
Let $p,q,r\geq 2$.  Following Ringel's presentation
\citep{Ringel1984}, let $\Delta(p,q,r)$ be the quiver consisting of
three paths with a common source $s$ and a common sink $t$:
\[
\xymatrix@R=1.4em@C=2.2em{
    & x_{1,1}\ar[r] & {\cdots}\ar[r] & x_{1,p-1}\ar[dr] & \\
    s\ar[ur]\ar[r]\ar[dr] & x_{2,1}\ar[r] & {\cdots}\ar[r]
        & x_{2,q-1}\ar[r] & t \\
    & x_{3,1}\ar[r] & {\cdots}\ar[r] & x_{3,r-1}\ar[ur] &
}
\]
Thus the three paths have lengths $p,q,r$, respectively, and are otherwise
disjoint. When an arm has length two, its first and last interior vertices
coincide, so that arm has a single interior vertex.  Denote the corresponding paths from $s$ to $t$ by
$\omega^{(1)},\omega^{(2)},\omega^{(3)}$.  The three-arm canonical algebra is
\[
  \Lambda(p,q,r)
  :=
  \K\Delta(p,q,r)\big/
  \big\langle\omega^{(1)}+\omega^{(2)}+\omega^{(3)}\big\rangle.
\]
We use $\Delta(p,q,r)$ for the quiver and $\Lambda(p,q,r)$ for the
resulting algebra throughout. The displayed relation is a normalized
form of the generic relation among the three paths.
\end{definition}

\begin{theorem}\label{thm:rank_two_canonical}
Let $\cA$ be a central essential arrangement of three distinct lines in
$\R^2$.  Then
\[
    T_{\cA} \cong \Lambda(2,2,2).
\]
\end{theorem}

\begin{proof}
We first identify the bound quiver algebra underlying $T_{\cA}$. Since
$\cA$ consists of three distinct lines through the origin and is
essential, its intersection lattice consists of $\R^2$, $H_1$, $H_2$, $H_3$,
and $\{0\}$. Recall the discussion in
\cref{subsec:Saliola}. There is no element of $L(\cA)$ strictly
between $\R^2$ and $H_i$, or between $H_i$ and $\{0\}$.  Hence the cover
relations in the inclusion order are $\R^2\gtrdot H_i\gtrdot\{0\}$ for $i=1,2,3$.
Thus the Hasse diagram of $L(\cA)^{\mathrm{op}}$ is a three-arm diagram. With the
orientation convention used for Tits algebras, this gives the quiver
\[
\xymatrix@R=1.8em@C=2.3em{
    & s\ar[dl]_{\alpha_1}\ar[d]^{\alpha_2}\ar[dr]^{\alpha_3} & \\
    x_1\ar[dr]_{\beta_1} & x_2\ar[d]_{\beta_2} & x_3\ar[dl]^{\beta_3} \\
    & t &
}
\]
where the vertices $s,x_i,t$ correspond respectively to
$\R^2,H_i,\{0\}$.

By \cref{lem:saliola_bound_quiver}, the relations are obtained from the intervals of
length two in the intersection lattice: for each such interval, one imposes the
sum of all length-two paths in that interval. In the displayed opposite order
there is only one interval of length two, namely $[\{0\},\R^2]$.
Its three saturated chains $\R^2\gtrdot H_i\gtrdot\{0\}$, for $i=1,2,3$,
correspond exactly to the three paths $\omega^{(i)}=\beta_i\alpha_i$
from the common source to the common sink.  Therefore
\[
    T_{\cA}
    \cong \K Q/\langle\beta_1\alpha_1+\beta_2\alpha_2+\beta_3\alpha_3\rangle
    =\K Q/\langle\omega^{(1)}+\omega^{(2)}+\omega^{(3)}\rangle.
\]
This is precisely the canonical algebra of type $(2,2,2)$: it has
three arms, each of length $2$, and one linear relation among the three maximal
paths from the source to the sink. Hence $T_{\cA}\cong\Lambda(2,2,2)$.
\end{proof}

To determine the representation type, we use Ringel's trichotomy for the
classical algebras $\Lambda(p,q,r)$~\citep{Ringel1984}.

\begin{theorem}[Ringel's trichotomy]\label{thm:ringel_trichotomy}
For the classical algebra $\Lambda(p,q,r)$, the following hold:
\begin{enumerate}
\item If $\frac1p+\frac1q+\frac1r>1$, then $\Lambda(p,q,r)$ is a
\emph{tame concealed algebra} and, in particular, \emph{domestic tame};
\item If $\frac1p+\frac1q+\frac1r=1$, then $\Lambda(p,q,r)$ is \emph{tubular};
\item If $\frac1p+\frac1q+\frac1r<1$, then it is \emph{wild}.
\end{enumerate}
\end{theorem}

We can now determine the representation type.
\begin{theorem}\label{cor:rank_two_tame}
For a central essential arrangement $\cA$ of three distinct lines in $\R^2$,
the Tits algebra $T_{\cA}$ is of tame representation type.
\end{theorem}
\begin{proof}
Since we already know that $T_{\cA}\cong \Lambda(2,2,2)$, we can apply Ringel's trichotomy.
In this case, $\frac12+\frac12+\frac12=\frac32>1$,
so $T_{\cA}$ is a tame concealed algebra and in particular has tame
representation type.
\end{proof}

\subsection{\texorpdfstring{The case \(m\geq4\)}{The case m >= 4}}

For this case, we embed the module category of the free algebra
$\K\langle x,y\rangle$ into $A_m\text{-mod}$ to show that $A_m$ is of wild representation type for $m\geq4$.

\begin{theorem}\label{th:rank_two_wild_embedding}
For every $m\ge 4$ there is a full exact embedding
\[
G:\K\langle x,y\rangle\text{-mod}\hookrightarrow A_m\text{-mod}.
\]
Consequently, $A_m$ is strictly wild for every $m\ge 4$.
\end{theorem}

\begin{proof}
Let $(M,X,Y)$ be a finite-dimensional $\K\langle x,y\rangle$-module. Define
an $A_m$-module $G(M,X,Y)$ by taking $V_s=V_{x_1}=\cdots=V_{x_m}=V_t=M$,
with arrow maps $\alpha_i=\id_M$ for $1\le i\le m$ and
\[
\begin{gathered}
\beta_1=X,\qquad\beta_2=Y,\qquad\beta_3=\id_M,\\
\beta_4=-(X+Y+\id_M),\qquad\beta_i=0\quad(i\ge5).
\end{gathered}
\]
Then
\[
\sum_{i=1}^m \beta_i\alpha_i
= X+Y+\id_M-(X+Y+\id_M)+0+\cdots+0=0,
\]
so the defining relation of $A_m$ is satisfied. Therefore $G(M,X,Y)$ is indeed an $A_m$-module.

Now let $T:(M,X,Y)\to(M',X',Y')$ be a morphism of
$\K\langle x,y\rangle$-modules, so $TX=X'T$ and $TY=Y'T$.
Assign the same linear map $T$ to every vertex of $Q_m$; then all squares commute, so we get a morphism $G(M,X,Y)\to G(M',X',Y')$.

Conversely, let
\[
\varphi=(f_s,f_1,\dots,f_m,f_t):G(M,X,Y)\to G(M',X',Y')
\]
be an $A_m$-morphism. Since all $\alpha_i$ are identities, commutativity at
these arrows implies $f_1=\cdots=f_m=f_s$.
Since $\beta_3=\id$, commutativity at $\beta_3$ implies $f_t=f_s$.
Finally, commutativity at $\beta_1$ and $\beta_2$ becomes
$f_sX=X'f_s$ and $f_sY=Y'f_s$.
Thus every $A_m$-morphism comes from a unique
$\K\langle x,y\rangle$-morphism. Hence $G$ is full and faithful, while
exactness is checked vertexwise. In particular, $G$ preserves
indecomposability and isomorphism classes, so $A_m$ is strictly wild.
\end{proof}

\subsection{Schurian finiteness}\label{sec:rank_two_schurian}

We can now determine the Schurian finiteness of the rank-two Tits algebras.

\begin{theorem}\label{thm:rank_two_schurian}
The rank-two Tits algebra $A_m$ is Schurian-finite if and only if $m=2$.
\end{theorem}

\begin{proof}
By \cref{thm:rank_two_finite}, the algebra $A_2$ is representation-finite,
so it is Schurian-finite.

Suppose first that $m=3$. Let
\[
    \mathsf K_2:\qquad
    \xymatrix@C=3.4em{
        1\ar@<0.55ex>[r]^x\ar@<-0.55ex>[r]_y & 2
    }
\]
be the $2$-Kronecker quiver. Write one of its finite-dimensional
representations as
\[
    M=\left(
    \xymatrix@C=3.4em{
        U\ar@<0.55ex>[r]^X\ar@<-0.55ex>[r]_Y & V
    }
    \right).
\]
Define an $A_3$-module $F_3(M)$ by
\[
\begin{gathered}
    F_3(M)_s=F_3(M)_{x_3}=U,
    \qquad
    F_3(M)_{x_1}=F_3(M)_{x_2}=F_3(M)_t=V,\\
    \alpha_1=X,\quad
    \alpha_2=Y,\quad
    \alpha_3=\id_U,\qquad
    \beta_1=\beta_2=\id_V,\quad
    \beta_3=-(X+Y).
\end{gathered}
\]
Then
\[
    \beta_1\alpha_1+\beta_2\alpha_2+\beta_3\alpha_3
    =X+Y-(X+Y)=0,
\]
so $F_3(M)$ satisfies the defining relation of $A_3$.

A morphism $(p,q):M\to M'$ of Kronecker representations satisfies
$qX=X'p$ and $qY=Y'p$.
Assigning $p$ at the vertices $s,x_3$ and $q$ at the vertices
$x_1,x_2,t$ defines an $A_3$-morphism. Conversely, the identity arrows in
$F_3(M)$ force every $A_3$-morphism to have precisely this form, and the two
remaining commutativity conditions are the Kronecker conditions above.
Consequently, $F_3:\operatorname{rep}_{\K}(\mathsf K_2)\hookrightarrow A_3\text{-mod}$
is full, faithful, and exact.

For each $\lambda\in\K$, set
\[
    M_\lambda=
    \left(
    \xymatrix@C=3.4em{
        {\K}\ar@<0.55ex>[r]^1\ar@<-0.55ex>[r]_\lambda & {\K}
    }
    \right).
\]
The modules $M_\lambda$ are pairwise non-isomorphic bricks. Since $\K$ is
infinite and $F_3$ is fully faithful, the modules $F_3(M_\lambda)$ form an
infinite family of pairwise non-isomorphic Schurian $A_3$-modules. Thus $A_3$
is Schurian-infinite.

Finally, let $m\geq4$. Regard $N_\lambda=(\K,\id_{\K},\lambda\id_{\K})$, for $\lambda\in\K$,
as a one-dimensional $\K\langle x,y\rangle$-module. These modules are
pairwise non-isomorphic bricks. The fully faithful functor $G$ from
\cref{th:rank_two_wild_embedding} preserves their endomorphism algebras and
isomorphism classes. Hence the $A_m$-modules $G(N_\lambda)$ give infinitely
many pairwise non-isomorphic bricks, and $A_m$ is Schurian-infinite.
\end{proof}
\section{\texorpdfstring{The $\R^3$ Case}{The R3 Case}}
\label{ch:dimension_three}

We retain the notation and conventions of \cref{ch:basic_results}.
In dimension three, we first treat three planes and then prove wildness for
four planes. Larger arrangements will be covered by deletion in
\cref{ch:general}.
\subsection{Three planes in \texorpdfstring{$\R^3$}{R3}}

For three distinct hyperplanes in $\R^n$, the intersection is at least $(n-3)$-dimensional. Hence, for $n\ge 4$, the intersection is at least one-dimensional, and the arrangement is not essential. Therefore, we only need to consider the case of three planes in $\R^3$.

\begin{theorem}\label{prop:three_planes_tame}
Let $\cA$ be a central essential arrangement of three distinct planes in
$\R^3$. Then
\(T_{\cA}\cong I(B_3)\),
where $I(B_3)$ is the incidence algebra of the Boolean lattice $B_3$, viewed as
the three-dimensional cube. Consequently, $T_{\cA}$ is of tame representation type.
\end{theorem}

\begin{proof}
For each $i=1,2,3$, choose a nonzero linear form
$\alpha_i\in(\R^3)^*$ such that $H_i=\ker(\alpha_i)$.
Essentiality implies that the three
covectors $\alpha_1,\alpha_2,\alpha_3$ form a basis of $(\R^3)^*$. Therefore
the linear isomorphism
\[
    \Phi:\R^3\longrightarrow\R^3,\qquad
    v\longmapsto\bigl(\alpha_1(v),\alpha_2(v),\alpha_3(v)\bigr)
\]
sends $H_i$ to the coordinate hyperplane
\[
    C_i:=\{(x_1,x_2,x_3)\in\R^3:x_i=0\},\qquad i=1,2,3.
\]
Thus it sends $\cA$ to the coordinate arrangement
$\cC_3:=\{C_1,C_2,C_3\}$ and induces an isomorphism of face semigroups, so
$T_{\cA}\cong T_{\cC_3}$,
as $\K$-algebras.

The face sign vectors of $\cC_3$ are all elements of $\{+,0,-\}^3$. Since
the coordinate arrangement is the product of three one-dimensional
arrangements, and the one-dimensional face algebra is the incidence algebra
of the two-element Boolean lattice $B_1$, we obtain
\[
    T_{\cC_3}
    \cong \K[\{+,0,-\}^3]
    \cong \K[\{+,0,-\}]^{\otimes3}
    \cong I(B_1)^{\otimes3}
    \cong I(B_3).
\]
The incidence algebra of the cube $B_3$ is of tame representation type~\citep[Lemma~3.5]{chapotonRognerud2019}. This proves
the theorem.
\end{proof}

We now address Schurian finiteness in the case of three planes.

\begin{theorem}\label{prop:three_planes_shurian_infinite}
Let $\cA$ be a central essential arrangement of three distinct planes in
$\R^3$. Then $T_{\cA}$ is Schurian-infinite.
\end{theorem}

\begin{proof}
Let $L_{ij}=H_i\cap H_j$ for $1\leq i<j\leq3$. Essentiality implies
that the three lines $L_{12},L_{13},L_{23}$ are distinct. The set
$S=\{H_1,H_2,H_3,L_{12},L_{13},L_{23}\}$ consists of two consecutive
ranks of the intersection lattice. Thus it is convex and contains no
interval of length two. By \cref{cor:convex-subquiver}, it suffices
to construct infinitely many pairwise non-isomorphic bricks for the path
algebra $B=\K Q_S$ of the following quiver:
\[
\xymatrix@R=2.0em@C=2.7em{
    H_1\ar[r]\ar[d] & L_{12} & H_2\ar[l]\ar[d] \\
    L_{13} & H_3\ar[l]\ar[r] & L_{23}
}
\]
For each $\lambda\in\K$, define $M_\lambda$ by assigning $\K$ to each
vertex, the scalar $\lambda$ to $H_1\to L_{13}$, and the identity map
to the other five arrows. Those five identity arrows form a connected
spanning path, so the commutativity conditions force every endomorphism
to act by the same scalar at all six vertices. Conversely, every such
scalar defines an endomorphism. Hence
$\End_B(M_\lambda)\cong\K$.

If $M_\lambda\cong M_\mu$, the same five arrows force the components of
an isomorphism to be a common nonzero scalar $c$. The remaining arrow
gives $c\lambda=\mu c$, so $\lambda=\mu$. Since $\K$ is infinite,
these modules form an infinite family of pairwise non-isomorphic bricks.
\end{proof}

\subsection{Four planes in \texorpdfstring{$\R^3$}{R3}}

We now investigate the case of four distinct planes.

\begin{theorem}\label{thm:four_planes_r3_wild}
Let $\cA$ be a central essential arrangement of four distinct planes in
$\R^3$. Then $T_{\cA}$ is of wild representation type.
\end{theorem}

\begin{proof}
Let $\cA=\{H_1,H_2,H_3,H_4\}$ be central and essential in $\R^3$.
Any two distinct planes meet in a line, so the intersection of any three is
either a line or the origin. At most one triple can meet in a line: two
distinct triples share two planes, and both triple-intersection lines would
have to equal the intersection of those two planes. That line would then lie
in all four planes, contradicting essentiality. We therefore have exactly
the two cases below.

In either case, let $S$ consist of the four planes and all their distinct
pairwise intersection lines, and let $Q_S$ be the induced quiver.
Its arrows run from each plane to the lines it contains. These vertices
occupy two consecutive ranks, so $S$ is convex and contains no interval of
length two. By \cref{cor:convex-subquiver}, it remains to show that
$\K Q_S$ is wild.

\medskip
\noindent\textbf{Case 1: every three planes meet only at the origin.}
The six lines $H_i\cap H_j$, with $i<j$, are pairwise distinct. Indeed, an
equality between two of them would give a line contained in at least three
planes. Each line is contained in exactly its two defining planes, since a
third containing plane would give the same contradiction. Thus the
intersection lattice consists of $\R^3$, four planes, six lines, and $\{0\}$.
The underlying unoriented graph of $Q_S$ is the subdivision graph of $K_4$:
the vertex corresponding to $H_i\cap H_j$ subdivides the edge joining
$H_i$ and $H_j$. Therefore
$|V(Q_S)|=10$, $|E(Q_S)|=12$, and $\beta(Q_S)=3$.
This connected graph is neither Dynkin nor extended Dynkin. By the
hereditary trichotomy~\citep{Ringel1984}, $\K Q_S$ is wild, and hence
$T_{\cA}$ is wild.

\medskip
\noindent\textbf{Case 2: three planes share a common line.}
After relabelling, let $L=H_1\cap H_2\cap H_3$ be a line.
The pairwise intersections among $H_1,H_2,H_3$ all equal $L$, since each is
a line containing $L$. Essentiality gives $L\cap H_4=\{0\}$.
For $i=1,2,3$, set $M_i=H_i\cap H_4$; these are lines distinct from $L$.
They are also pairwise distinct: if $M_i=M_j$ for $i\neq j$, that line
would lie in $H_i\cap H_j\cap H_4=L\cap H_4=\{0\}$, a contradiction.
The same argument shows that $M_i$ is not contained in $H_j$ for
$j\in\{1,2,3\}\setminus\{i\}$. Thus $L$ lies in exactly $H_1,H_2,H_3$,
and each $M_i$ lies in exactly $H_i$ and $H_4$.
The vertices of $Q_S$ are consequently $H_1,H_2,H_3,H_4$ and
$L,M_1,M_2,M_3$.
The induced quiver has the arrows
\[
\xymatrix@R=1.8em@C=3.0em{
    & M_1 & H_1\ar[l]\ar[dr] & \\
    H_4\ar[ur]\ar[r]\ar[dr] & M_2 & H_2\ar[l]\ar[r] & L \\
    & M_3 & H_3\ar[l]\ar[ur] &
}
\]
Consequently,
$|V(Q_S)|=8$, $|E(Q_S)|=9$, and $\beta(Q_S)=9-8+1=2$.
This connected graph is again neither Dynkin nor extended Dynkin. Thus
$\K Q_S$, and consequently $T_{\cA}$, is wild. This completes the proof.
\end{proof}
\section{\texorpdfstring{The $\R^4$ Case}{The R4 Case}}
\label{ch:dimension_four}

We continue with four distinct hyperplanes. Their common intersection in
$\R^n$ has dimension at least $n-4$, so essentiality forces $n\leq4$.
The cases of dimensions two and three have been treated in
\cref{ch:rank_two,ch:dimension_three}; it remains to
consider dimension four.

\begin{theorem}\label{thm:four_hyperplanes_r4_wild}
Let $\cA$ be a central essential arrangement of four distinct hyperplanes
in $\R^4$. Then $T_{\cA}$ is of wild representation type.
\end{theorem}

\begin{proof}

By \cref{lem:saliola_bound_quiver}, the vertices of
the Gabriel quiver of $T_{\cA}$ are the elements of the intersection lattice
$L(\cA)$, ordered by reverse inclusion, and there is an arrow
\(X\longrightarrow Y\)
whenever $Y$ covers $X$. The defining relations come from intervals of length
two.

Write $\cA=\{H_1,H_2,H_3,H_4\}$, where $H_i=\ker(\alpha_i)$ and
$\alpha_i\in(\R^4)^*$ is a nonzero linear form. By essentiality,
the four covectors span $(\R^4)^*$
and therefore form a basis. Intersections indexed by distinct
subsets of these four hyperplanes are distinct, and hence
$L(\cA)\cong B_4$,
the Boolean lattice of rank four.

Take the vertices in ranks one and two:
\(S=\{H_1,H_2,H_3,H_4\}\cup\{H_i\cap H_j:1\leq i<j\leq4\}\).
Since these are two consecutive ranks, $S$ is convex and contains no interval
of length two. By \cref{cor:convex-subquiver}, it suffices to prove that
the path algebra of the induced quiver $Q_S$ is wild.
For every $i<j$, the induced quiver has the two arrows
\[
\xymatrix@R=1.3em@C=1.8em{
    H_i\ar[dr] && H_j\ar[dl] \\
    & H_i\cap H_j &
}
\]
Its underlying unoriented graph is the subdivision graph of $K_4$. Thus
$|V(Q_S)|=4+6=10$, $|E(Q_S)|=12$, and its cyclomatic number is
$\beta(Q_S)=|E(Q_S)|-|V(Q_S)|+1=3$.
The graph is therefore neither Dynkin nor extended Dynkin. By the hereditary
trichotomy~\citep{Ringel1984}, the path algebra $\K Q_S$ is wild, and hence
$T_{\cA}$ is wild.
\end{proof}
\section{General Case}\label{ch:general}
In this section, we shall prove that any central arrangement can be reduced to central essential arrangements, and moreover, for any central essential arrangement with at least four hyperplanes, the corresponding Tits algebra is wild.

For a central arrangement $\cA$ in $\R^n$, the common intersection
$X=\bigcap_{H\in\cA}H$ consists of directions along which none of the
defining linear forms changes. Translating a point along $X$ therefore
leaves its face sign vector unchanged. This suggests removing those common
directions by passing to $\R^n/X$ and replacing each hyperplane $H$ by $H/X$;
the resulting hyperplanes have zero total intersection. This reduction is
called \emph{essentialization}. The following lemma shows that it preserves
the Tits algebra, so the representation-theoretic classification can be
carried out for essential arrangements.

\begin{lemma}[Essentialization]\label{lem:essentialization}
Let $\cA=\{H_1,\dots,H_m\}$ be a central arrangement in $\R^n$, and set
\[
    X=\bigcap_{i=1}^m H_i,
    \qquad
    r=n-\dim X,
    \qquad
    \overline H_i=H_i/X.
\]
Then
$\cA_{\mathrm{ess}}:=\{\overline H_1,\dots,\overline H_m\}$
is an essential arrangement in $\R^n/X$, which we identify with $\R^r$,
and
\[
    T_{\cA}\cong T_{\cA_{\mathrm{ess}}}.
\]
Consequently, $\cA$ and $\cA_{\mathrm{ess}}$ have the same representation
type and the same Schurian finiteness.
\end{lemma}

\begin{proof}
Write $H_i=\ker(\alpha_i)$ with $\alpha_i\in(\R^n)^*$ a nonzero linear
form. Since $X\subseteq H_i$,
each $\alpha_i$ vanishes on $X$ and descends to a functional
$\overline\alpha_i\in(\R^n/X)^*$ with $\overline H_i=\ker(\overline\alpha_i)$.
Moreover,
\[
    \bigcap_{i=1}^m\overline H_i
    =\left(\bigcap_{i=1}^mH_i\right)/X
    =0,
\]
so the quotient arrangement is essential.

For $v\in\R^n$, its sign vector with respect to $\cA$ is
\[
    \operatorname{sgn}_{\cA}(v)
    =\bigl(
       \operatorname{sgn}\alpha_1(v),\dots,
       \operatorname{sgn}\alpha_m(v)
     \bigr).
\]
If $x\in X$, then $\alpha_i(v+x)=\alpha_i(v)$ for every $i$. Hence the sign
vector depends only on the coset $v+X$, and the two arrangements have the same
set of face sign vectors. Their products also agree, since the Tits product is
given by
\[
    (\sigma\circ\tau)_i=
    \begin{cases}
        \sigma_i, & \sigma_i\neq0,\\
        \tau_i,   & \sigma_i=0.
    \end{cases}
\]
Thus $\cF(\cA)\cong\cF(\cA_{\mathrm{ess}})$, and taking semigroup
algebras gives the claimed isomorphism.
\end{proof}

To pass from small arrangements to larger ones, we use deletion. The
following quotient construction preserves both wildness and Schurian
infiniteness in the required direction.

\begin{lemma}[Deletion]\label{lem:deletion_tits_algebra}
Let $\cA=\{H_1,\dots,H_m\}$ be a central arrangement in $\R^n$, and let
$\cB=\{H_{i_1},\dots,H_{i_k}\}\subseteq\cA$ be a subarrangement. Then there is a surjective algebra
homomorphism
\[
    T_{\cA}\twoheadrightarrow T_{\cB}.
\]
In particular, if $T_{\cB}$ is wild, then $T_{\cA}$ is wild; if
$T_{\cB}$ is Schurian-infinite, then $T_{\cA}$ is Schurian-infinite.
\end{lemma}

\begin{proof}
Restricting a sign vector to the coordinates indexed by $\cB$ defines the
map $\pi:\cF(\cA)\longrightarrow\cF(\cB)$.
It is surjective: for a face $G$ of $\cB$, choose any point $v\in G$ and take
the face of $\cA$ containing $v$; its sign vector restricts to that of $G$.
Since the Tits product is defined coordinatewise, $\pi$ is a semigroup
homomorphism. Its linear extension gives the required surjective algebra map.
The final assertion follows from the quotient-algebra inheritance in
\cref{lem:quotient_inheritance}.
\end{proof}

To conclude, we have the following theorem.
\begin{theorem}\label{cor:four_or_more_wild}
Let $\cA$ be a central arrangement containing at least four distinct
hyperplanes. Then $T_{\cA}$ is wild.
\end{theorem}

\begin{proof}
By \cref{lem:essentialization}, we may assume that $\cA$ is essential.
Essentialization does not identify distinct hyperplanes, since each $H_i$ is
the inverse image of $H_i/X$ under the quotient map. Let $r$ be the rank of
$\cA$, equivalently the dimension of the span of its defining covectors. Since
$\cA$ contains at least four distinct central hyperplanes, one has $r\geq2$.

If $r=2$, the result follows from the explicit embedding of
\cref{th:rank_two_wild_embedding}.

If $r=3$, choose three hyperplanes whose defining covectors are linearly
independent and adjoin any fourth distinct hyperplane. The resulting
subarrangement $\cB$ is essential in dimension three, so
$T_{\cB}$ is wild by \cref{thm:four_planes_r3_wild}.
The quotient map of \cref{lem:deletion_tits_algebra} then implies
that $T_{\cA}$ is wild. No general-position assumption on the four
hyperplanes is needed.

If $r\geq4$, choose four hyperplanes with linearly independent defining
covectors and let $\cB$ be the resulting rank-four subarrangement.
Its essentialization is an essential arrangement of four hyperplanes in
$\R^4$, so its Tits algebra is wild by
\cref{thm:four_hyperplanes_r4_wild}. Essentialization identifies
this algebra with $T_{\cB}$, and deletion gives a quotient
$T_{\cA}\twoheadrightarrow T_{\cB}$. Hence $T_{\cA}$ is wild in this
case as well.
\end{proof}

We now give the Schurian-finiteness criterion for all central arrangements,
including nonessential arrangements and those of rank zero or one.
\begin{theorem}\label{thm:schurian_finiteness}
Let $\cA$ be a central arrangement of distinct hyperplanes in $\R^n$.
Then $T_{\cA}$ is Schurian-finite if and only if $|\cA|\leq2$.
\end{theorem}

\begin{proof}
If $|\cA|=0$, its face semigroup has one element and $T_{\cA}\cong\K$.
If $|\cA|=1$, essentialization reduces to one hyperplane in $\R$;
by Saliola's presentation its algebra is the path algebra of a single
arrow, which is representation-finite. If $|\cA|=2$, the two defining
covectors are independent, so essentialization gives two distinct lines
in $\R^2$ and \cref{thm:rank_two_finite} applies. Thus in each
case $T_{\cA}$ is representation-finite and hence Schurian-finite.

Now suppose that $|\cA|\geq 3$, and choose any three-element subarrangement
$\cB=\{H_1,H_2,H_3\}\subseteq\cA$.
By \cref{lem:deletion_tits_algebra}, there is a surjective algebra
homomorphism $T_{\cA}\twoheadrightarrow T_{\cB}$. Since the three
hyperplanes are distinct, the rank of $\cB$ is either $2$ or $3$.
If $\operatorname{rk}(\cB)=2$, essentialization and
\cref{thm:rank_two_schurian} show that $T_{\cB}$ is Schurian-infinite.
If $\operatorname{rk}(\cB)=3$, essentialization and
\cref{prop:three_planes_shurian_infinite} show that $T_{\cB}$
is Schurian-infinite. Thus $T_{\cB}$ is Schurian-infinite in either case.
Since it is a quotient of
$T_{\cA}$, quotient inheritance implies that $T_{\cA}$ is
Schurian-infinite.
\end{proof}

\ifUseBibliography

\fi

\ifUseIndex
  \printindex
\fi


\begin{thebibliography}{10}
\providecommand{\url}[1]{\texttt{#1}}
\providecommand{\urlprefix}{URL }
\providecommand{\doi}[1]{https://doi.org/#1}

\bibitem{AdachiIyamaReiten2014}
Adachi, T., Iyama, O. and Reiten, I., {$\tau$}-tilting theory, Compos. Math.
  \textbf{150}(3),  415--452 (2014). \doi{10.1112/S0010437X13007422}.

\bibitem{AssemSimsonSkowronski2006}
Assem, I., Simson, D. and Skowro\'nski, A., Elements of the Representation
  Theory of Associative Algebras, {Vol. 1}: Techniques of Representation
  Theory, London Math. Soc. Student Texts, vol.~65. Cambridge University Press,
  Cambridge (2006). \doi{10.1017/CBO9780511614309}.

\bibitem{Benson1991}
Benson, D.~J., Representations and Cohomology {I}: Basic Representation Theory
  of Finite Groups and Associative Algebras, Cambridge Studies in Advanced
  Mathematics, vol.~30. Cambridge University Press, Cambridge (1991).
  \doi{10.1017/CBO9780511623615}.

\bibitem{brownDiaconis1998}
Brown, K.~S. and Diaconis, P., Random walks and hyperplane arrangements, Ann.
  Probab.  \textbf{26}(4),  1813--1854 (1998).

\bibitem{chapotonRognerud2019}
Chapoton, F. and Rognerud, B., On the wildness of {Cambrian} lattices, Algebr.
  Represent. Theory  \textbf{22}(3),  603--614 (2019).
  \doi{10.1007/s10468-018-9789-1}.

\bibitem{DemonetIyamaJasso2019}
Demonet, L., Iyama, O. and Jasso, G., {$\tau$}-tilting finite algebras, bricks,
  and {$g$}-vectors, Int. Math. Res. Not. IMRN  \textbf{2019}(3),  852--892
  (2019). \doi{10.1093/imrn/rnx135}.

\bibitem{erdmannLim2025descent}
Erdmann, K. and Lim, K.~J., The representation type of the descent algebras of
  type {$\mathbb{A}$}, Sci. China Math.  \textbf{68}(12),  2827--2846 (2025).
  \doi{10.1007/s11425-024-2378-y}.

\bibitem{erdmannLim2026descent}
Erdmann, K. and Lim, K.~J., The representation type of the descent algebras
  (2026). \doi{10.48550/arXiv.2604.21619}, preprint, arXiv:2604.21619
  [math.RT].

\bibitem{GeigleLenzing1987}
Geigle, W. and Lenzing, H., A class of weighted projective curves arising in
  representation theory of finite dimensional algebras. In: Singularities,
  Representation of Algebras, and Vector Bundles, Lecture Notes in Math.,
  vol.~1273, pp. 265--297. Springer, Berlin (1987). \doi{10.1007/BFb0078849}.

\bibitem{Ringel1984}
Ringel, C.~M., Tame Algebras and Integral Quadratic Forms, Lecture Notes in
  Math., vol.~1099. Springer, Berlin (1984). \doi{10.1007/BFb0072870}.

\bibitem{saliola2009face}
Saliola, F.~V., The face semigroup algebra of a hyperplane arrangement, Canad.
  J. Math.  \textbf{61}(4),  904--929 (2009). \doi{10.4153/CJM-2009-046-2}.

\bibitem{Schiffler2014}
Schiffler, R., Quiver Representations. CMS Books in Mathematics, Springer, Cham
  (2014). \doi{10.1007/978-3-319-09204-1}.

\bibitem{schocker2006SolomonTits}
Schocker, M., The module structure of the {Solomon--Tits} algebra of the
  symmetric group, J. Algebra  \textbf{301}(2),  554--586 (2006).
  \doi{10.1016/j.jalgebra.2005.07.032}.

\bibitem{Stanley2007Arrangements}
Stanley, R.~P., An introduction to hyperplane arrangements. In: Miller, E.,
  Reiner, V. and Sturmfels, B. (eds.) Geometric Combinatorics, IAS/Park City
  Math. Ser., vol.~13, pp. 389--496. American Mathematical Society, Providence,
  RI (2007). \doi{10.1090/pcms/013/08}.

\end{thebibliography}
\end{document}